\documentclass[11pt]{article}

\usepackage{amsfonts,latexsym,amsthm,amssymb,amsmath,amscd,euscript,tikz,mathtools}
\usepackage{framed}
\usepackage{fullpage}
\usepackage{color}
\usepackage[colorlinks=true,citecolor=blue,linkcolor=blue]{hyperref}
\usepackage{enumitem}
\usepackage{siunitx}
\usepackage{textcomp}
\usepackage{physics}
\usepackage{mathrsfs}
\usepackage{dsfont}
\usepackage[ruled,vlined]{algorithm2e}
\usepackage{titlesec}
\usepackage{tikz-cd}
\usepackage{thmtools}
\usepackage{thm-restate}
\usepackage{cleveref}
\usepackage{graphicx}
\usetikzlibrary{positioning,chains,fit,shapes,calc}

\allowdisplaybreaks[1]

\newtheorem{theorem}{Theorem}
\newtheorem{proposition}[theorem]{Proposition}
\newtheorem{lemma}[theorem]{Lemma}
\newtheorem{claim}[theorem]{Claim}

\theoremstyle{definition}
\newtheorem{definition}[theorem]{Definition}

\newtheorem{remark}[theorem]{Remark}

\theoremstyle{remark}

\newcommand{\cG}{\mathcal{G}}

\newcommand{\cX}{\mathcal{X}}

\newcommand{\bN}{\mathbb{N}}

\newcommand{\bR}{\mathbb{R}}

\newcommand{\poly}{\operatorname{poly}}

\newcommand{\nc}{\newcommand}

\nc{\on}{\operatorname}
\nc{\Spec}{\on{Spec}}
\nc{\Aut}{\textit{Aut}}
\nc{\id}{\textit{id}}
\nc{\chr}{\on{char}}
\nc{\im}{\on{im}}
\nc{\Hom}{\on{Hom}}
\nc{\lcm}{\on{lcm}}
\nc{\dual}[1]{\prescript{t}{}{#1}}
\nc{\transpose}[1]{{#1}^{\intercal}}
\nc{\Sym}{\on{Sym}}
\nc{\End}{\on{End}}
\nc{\stab}{\on{stab}}
\nc{\Li}{\on{Li}}
\nc{\spn}{\on{span}}
\nc{\sgn}{\on{sgn}}
\nc{\supp}{\on{supp}}
\nc{\Unif}{\on{Unif}}

\title{New Explicit Constant-Degree Lossless Expanders}
\author{Louis Golowich\thanks{UC Berkeley. Email: \texttt{lgolowich@berkeley.edu}. This work was supported by a National Science Foundation Graduate Research Fellowship under Grant No.~DGE 2146752, and supported in part by V.~Guruswami's Simons Investigator award and UC Berkeley Initiative for Computational Transformation award.}}

\begin{document}


\maketitle


\begin{abstract}
  We present a new explicit construction of onesided bipartite lossless expanders of constant degree, with arbitrary constant ratio between the sizes of the two vertex sets. Our construction is simpler to state and analyze than the only prior construction of Capalbo, Reingold, Vadhan, and Wigderson (2002), and achieves improvements in some parameters.

  We construct our lossless expanders by imposing the structure of a constant-sized lossless expander ``gadget'' within the neighborhoods of a large bipartite spectral expander; similar constructions were previously used to obtain the weaker notion of unique-neighbor expansion. Our analysis simply consists of elementary counting arguments and an application of the expander mixing lemma.
\end{abstract}




\pagenumbering{arabic}

\section{Introduction}
\label{sec:intro}
We construct infinite families of constant-degree onesided lossless bipartite expanders with arbitrary constant ratio between the sizes of the left and right vertex sets. A lossless expander is defined as a graph for which for all sufficiently small vertex sets, most of the outgoing edges lead to distinct vertices. These objects are applicable to various areas of computer science, including networks and distributed algorithms \cite{peleg_constructing_1989,arora_-line_1996,broder_static_1999,makur_broadcasting_2020}, compressed sensing \cite{xu_efficient_2007,jafarpour_efficient_2009,indyk_near-optimal_2008}, error-correcting codes \cite{sipser_expander_1996,lin_c3-locally_2022,lin_good_2022}, and proof complexity \cite{ben-sasson_short_2001,alekhnovich_pseudorandom_2004,alekhnovich_lower_2001}, among others.

While constant-degree random graphs give lossless expanders with high probability, the only previously known explicit construction\footnote{Shortly after our paper was posted online, a similar result obtained independently and concurrently was posted in \cite{cohen_hdx_2023}; see Remark~\ref{remark:concurrent}.} was obtained by Capalbo, Reingold, Vadhan, and Wigderson \cite{capalbo_randomness_2002}, using a fairly involved form of the zigzag product \cite{reingold_entropy_2002}. Therefore given the numerous applications of lossless expanders described above, it is desirable to have additional, simpler explicit constructions. A new such construction, which simplifies the construction of \cite{capalbo_randomness_2002}, is the main result of this paper.

Unlike lossless expanders, there are numerous known explicit constructions of spectral expanders (e.g.~\cite{margulis_explicit_1973,lubotzky_ramanujan_1988,morgenstern_existence_1994,reingold_entropy_2002,ben-aroya_combinatorial_2011,kaufman_construction_2018}), which are defined as graphs with no large nontrivial eigenvalues of the adjacency matrix. Yet Kahale~\cite{kahale_eigenvalues_1995} showed that even optimal spectral expanders can fail to exhibit lossless expansion. Hence constructions of lossless expanders must rely on different techniques.

We now formally define lossless expansion.

\begin{definition}
  For real numbers $0\leq\mu,\epsilon\leq 1$, a bipartite graph $G=(L(G)\sqcup R(G),E(G))$ with left-degree $d$ is a (onesided) \textbf{$(\mu,\epsilon)$-lossless expander} if for every $S\subseteq L(G)$ with $|S|\leq\mu|L(G)|$, it holds that $|N_G(S)|\geq(1-\epsilon)d|S|$.
\end{definition}

The main result of \cite{capalbo_randomness_2002}, which we recover with a simpler construction and analysis, is stated below.

\begin{theorem}
  \label{thm:expanders}
  For every open interval $\beta=(\beta^{(1)},\beta^{(2)})\subseteq\bR_+$ and every $\epsilon>0$, there exists a sufficiently large $D=D(\beta,\epsilon)\in\bN$ and a sufficiently small $\mu=\mu(\beta,\epsilon)>0$ such that there is an infinite explicit family of $(\mu,\epsilon)$-lossless expanders $G$ with left-degree $D$ and with $|R(G)|/|L(G)|\in(\beta^{(1)},\beta^{(2)})$.
\end{theorem}

For fixed $\beta,\epsilon$, our construction achieves a smaller degree $D$ than \cite{capalbo_randomness_2002}, but \cite{capalbo_randomness_2002} achieves a larger expansion cutoff $\mu$ (see Remark~\ref{remark:compareparam}). Thus the two constructions have incomparable parameters. For instance, for a fixed constant $\beta$, we obtain $D=\tilde{O}(1/\epsilon^2)$ and $\mu=\tilde{\Omega}(\epsilon^{10})$, whereas \cite{capalbo_randomness_2002} obtains $D=\tilde{O}(1/\epsilon^3)$ and $\mu=\tilde{\Omega}(\epsilon^4)$. Neither construction achieves the optimal dependencies $D=\tilde{O}(1/\epsilon)$ and $\mu=\tilde{\Omega}(\epsilon^2)$ (see Proposition~\ref{prop:randloss} below).

Reducing the degree $D$ as a function of $\epsilon$ is useful in recent applications of lossless expanders \cite{guruswami_lp-spread_2022,doron_almost_2023}. For instance, by Corollary~10 of \cite{guruswami_lp-spread_2022}, our improved dependence $D=\tilde{O}(1/\epsilon^2)$ immediately yields explicit constructions of matrices with the $\ell_p$-restricted isometry property for new values of $p$; such matrices have applications to compressed sensing (see e.g.~\cite{allen-zhu_restricted_2015}).

We remark that while Theorem~\ref{thm:expanders} provides explicit onesided lossless expansion, there remains no known explicit construction of twosided lossless expanders, which require lossless expansion from right to left as well as from left to right.

Theorem~\ref{thm:expanders} provides constant-degree lossless expanders with any constant ratio $|R(G)|/|L(G)|=\Theta(1)$ between the sizes of the left and right vertex sets. Highly unbalanced lossless expanders, where $|R(G)|/|L(G)|$ decays polynomially and the left-degree grows poly-logarithmically with the number of vertices, were constructed by \cite{ta-shma_lossless_2007,guruswami_unbalanced_2009,kalev_unbalanced_2022} to obtain randomness extractors.

Our construction follows the same framework as the unique-neighbor expander constructions of \cite{alon_explicit_2002,becker_symmetric_2016,asherov_bipartite_2023,hsieh_explicit_2023} in that we begin with a good spectral expander and then impose the structure of a smaller ``gadget graph'' locally in the neighborhoods of vertices in the spectral expander. In fact, our construction is essentially the same as the onesided unique neighbor expanders of \cite{asherov_bipartite_2023}, though we provide a new analysis in order to obtain the stronger object of onesided lossless expanders. Note that unlike lossless expanders, which require most vertices in the neighborhood of every sufficiently small set $S$ to be connected to $S$ by a unique edge, unique-neighbor expanders only require the existence of a single such vertex\footnote{Some definitions of unique-neighbor expansion make the stronger requirement that at least a small constant fraction of the neighbors of $S$ are connected to $S$ by a single edge. However, this requirement is still weaker than lossless expansion.} in the neighborhood of $S$. Furthermore, whereas the analysis of \cite{asherov_bipartite_2023} requires the construction to be instantiated with an unbalanced bipartite Ramanujan graph with the optimal twosided expansion (of which few constructions are currently known), our analysis is more robust in that near-Ramanujan onesided expansion suffices (see Remark~\ref{remark:robustexp} for details).

We construct our lossless expanders $G$ by combining a large unbalanced bipartite spectral expander $X$ with a constant-sized lossless expander (a ``gadget'') $G_0$ as follows. We let $L(G)=L(X)$ and $R(G)=R(X)\times R(G_0)$, and then let $G$ be the union of $|R(X)|$ copies of the gadget $G_0$. Specifically, we add to $G$ a copy of $G_0$ for each $v\in R(X)$ by associating neighbors of $v$ with left-vertices of $G_0$, and elements of $\{v\}\times R(G_0)$ with right-vertices of $G_0$. Note that this construction requires the right-degree of $X$ to equal $|L(G_0)|$.

Whereas our construction combines a large unbalanced spectral expander with a small lossless expander, the construction of \cite{capalbo_randomness_2002} combines a large balanced spectral expander with two small gadgets, namely one lossless expander and one more sophisticed object called a ``buffer conductor.'' The unbalanced spectral expander in our construction essentially serves the same purpose as the combination of the balanced spectral expander and buffer conductor in \cite{capalbo_randomness_2002}.

We remark that both our construction and that of \cite{capalbo_randomness_2002} yield graphs permitting a free group action on the vertices and edges, with group size linear in the number of vertices. Onesided lossless expanders permitting such group actions in turn give asymptotically good locally testable codes by \cite{lin_c3-locally_2022}. Thus our construction implies a new family of good locally testable codes.

Similarly, \cite{lin_good_2022} show that \textit{twosided} lossless expanders permitting a group action imply asymptotically good quantum LDPC codes with linear-time decoders. While no such expanders are currently known, it is an interesting question whether our techniques could be extended to obtain twosided expanders that can instantiate these codes. However, we note that there are other unconditional constructions of good quantum LDPC codes \cite{panteleev_asymptotically_2021,leverrier_quantum_2022-1,dinur_good_2023} with linear-time decoders \cite{leverrier_efficient_2023,gu_efficient_2023,dinur_good_2023}.


\begin{remark}
  \label{remark:concurrent}
  In independent and concurrent work, Cohen, Roth, and Ta-Shma \cite{cohen_hdx_2023} obtained a similar construction of lossless expanders. Specifically, \cite{cohen_hdx_2023} and our work both use the same framework described above of combining a large unbalanced bipartite spectral expander with a constant-sized lossless expander. However, \cite{cohen_hdx_2023} construct the large bipartite spectral expander using the hyper-regular high-dimensional expanders (HDXs) of \cite{friedgut_hyper-regular_2020}. In contrast, we show that it suffices to use HDXs with weaker regularity properties, or to simply use bipartite Ramanujan graphs.
\end{remark}

\section{Preliminaries}
This section introduces basic notions and known results.

For a graph $G=(V(G),E(G))$ and a set of vertices $S\subseteq V(G)$, we let $N_G(S)$ denote the set of neighbors of $S$ in $G$. For $v,v'\in V(G)$, we let $w_G(v,v')$ denote the weight of the edge from $v$ to $v'$ (which for a simple graph is always $0$ or $1$). Similarly, for $S,S'\subseteq V(G)$, we let $w_G(S,S')=\sum_{(v,v')\in S\times S'}w_G(v,v')$ denote the sum of the weights of edges from vertices in $S$ to vertices in $S'$. For a vertex $v\in V(G)$, the degree $\deg(v)=w_G(v,V(G))$ equals the sum of the weights of the edges incident to that vertex. For a bipartite graph $G$, we let $V(G)=L(G)\sqcup R(G)$ denote the decomposition into the left and right vertex sets.

The main focus of our paper is to construct lossless expanders satisfying the following standard notion of explicitness.

\begin{definition}
  \label{def:explicit}
  A family of graphs is \textbf{explicit} if there exists a $\poly(n)$-time algorithm that takes as input an integer $n$, and outputs an $n$-vertex graph in the family, if one exists. 
\end{definition}

Our analysis will rely heavily on the notion of spectral expansion, defined below.

\begin{definition}
  For an $n$-vertex graph $G$, the \textbf{(onesided) spectral expansion} $\lambda_2(G)$ is defined as the second largest eigenvalue of the random walk matrix of $G$. Formally, letting $W_G$ denote the random walk matrix, so that $(W_G)_{v,v'}=w_G(v,v')/\deg(v)$, if the eigenvalues of $W_G$ are $1=\lambda_1(W_G)\geq\lambda_2(W_G)\geq\cdots\geq\lambda_n(W_G)$, then $\lambda_2(G):=\lambda_2(W_G)$ is the onesided spectral expansion.
\end{definition}

We will make use of the following well known property of spectral expanders.

\begin{lemma}[Expander Mixing Lemma; see for instance Lemma~4.15 of \cite{vadhan_pseudorandomness_2012}] 
  \label{lem:expmix}
  For a $D$-regular graph $G$, it holds for every subset of vertices $S\subseteq V(G)$ that
  \begin{equation*}
    w_G(S,S) \leq \left(\lambda_2(G)+\frac{|S|}{|V(G)|}\right)D|S|.
  \end{equation*}
\end{lemma}

We will make use of unbalanced bipartite graphs for which the ``nonlazy'' or ``nonbacktracking'' length-2 walk (that is, the square) has good spectral expansion.

\begin{definition}
  For a bipartite graph $G$, the \textbf{nonlazy square} $G'$ is the graph on vertex set $V(G')=R(G)$, with edge weights given for $v,v'\in R(G)$ by $w_{G'}(v,v')=0$ if $v=v'$ and $w_{G'}(v,v')=\sum_{w\in L(G)}w_G(v,w)w_G(w,v')$ if $v\neq v'$.
\end{definition}

\begin{proposition}
  \label{prop:nonlazyexp}
  For every integer $k\geq 2$ and for every $\lambda_2\geq 0$, it holds for infinitely many $D\in\bN$ that there exists an infinite explicit family of $(k,D)$-biregular bipartite graphs whose nonlazy square has (onesided) spectral expansion $\leq\lambda_2$.
\end{proposition}
\begin{proof}
  We describe two different known constructions that each prove the proposition:
  \begin{enumerate}
  \item\label{it:hdxproof} If $X$ is a $(k-1)$-dimensional simplicial complex for $k\in\bN$ and $G$ is the incidence graph between $(k-1)$-dimensional faces $X(k-1)=L(G)$ and vertices $X(0)=R(G)$, then the nonlazy square $G'$ of $G$ is the 1-skeleton of $X$. For any fixed $k\in\bN$ and $\lambda_2>0$, Ramanujan complexes \cite{lubotzky_ramanujan_2005,lubotzky_explicit_2005} as well as the coset complexes of \cite{kaufman_construction_2018,odonnell_high-dimensional_2022-1} provide examples of explicit such $(k-1)$-dimensional simplicial complexes $X$ with constant degree and arbitrarily good spectral expansion $\lambda_2(G')\leq\lambda_2$.
  \item\label{it:bipramproof} Let $G$ be a $(k,D)$-biregular bipartite Ramanujan graph, for instance as constructed explicitly in \cite{gribinski_existence_2021}, so that every nontrivial eigenvalue of the unnormalized adjacency matrix $M_G$ of $G$ is at most $\lambda_2(M_G)\leq\sqrt{D-1}+\sqrt{k-1}$. We emphasize that here $(M_G)_{v,v'}=w_G(v,v')$ is unnormalized, so $\lambda_2(M_G)\neq\lambda_2(G)=\lambda_2(W_G)$. Then every nontrivial eigenvalue of the unnormalized adjacency matrix $M_{G^2}$ of the (ordinary) square $G^2$ is at most $\lambda_2(M_G)^2\leq(\sqrt{D-1}+\sqrt{k-1})^2=(D-1)+(k-1)+2\sqrt{(D-1)(k-1)}$. Here we have used the fact that the spectrum of $M_G$ is symmetric about $0$ as $G$ is bipartite. Furthermore, $M_{G^2}$ is block diagonal with blocks $L(G)\times L(G)$ and $R(G)\times R(G)$ both having spectrum equal to the square of the spectrum of $M_G$, up to zero-eigenvalues. Therefore $\lambda_2(M_{G^2}|_{R(G)})=\lambda_2(M_G)^2$. Thus every nontrivial eigenvalue of the unnormalized adjacency matrix $M_{G'}$ of the nonlazy square $G'$ is at most $\lambda_2(M_{G'})=\lambda_2(M_{G^2}|_{R(G)})-D\leq(k-1)+2\sqrt{(D-1)(k-1)}$. As $G'$ is $D(k-1)$-regular, it holds that $W_{G'}=M_{G'}/D(k-1)$, so $\lambda_2(G')=\lambda_2(W_{G'})\leq 1/D+2/\sqrt{D(k-1)}$. Thus $\lambda_2(G')\leq\lambda_2$ if $D$ is sufficiently large.
  \end{enumerate}
\end{proof}

\begin{remark}
  As Ramanujan complexes are Cayley complexes, the construction in the \ref{it:hdxproof}st proof of Proposition~\ref{prop:nonlazyexp} has the added benefit of permitting a free group action on the vertices and faces that acts transitively on the vertices. Our entire construction can be made to respect this group action, and the orbits have linear size with respect to the number of vertices. Hence our construction can be used to instantiate the asymptotically good locally testable codes of \cite{lin_c3-locally_2022}. 
\end{remark}

\begin{remark}
  We prove Proposition~\ref{prop:nonlazyexp} for the ordinary notion of (weak) explicitness given in Definition~\ref{def:explicit}. However, we could instead consider \textit{strong explicitness}, which requires the family of graphs to have a $\poly(\log n)$-time algorithm that takes as input integers $n,i,j$, and outputs the $j$th neighbor of the $i$th vertex of the $n$-vertex graph in the family, if one exists. Some of our constructions proving Proposition~\ref{prop:nonlazyexp}, such as the construction using the high-dimensional expanders of \cite{kaufman_construction_2018,odonnell_high-dimensional_2022-1}, satisfy this notion of strong explicitness (see \cite{odonnell_high-dimensional_2022-1} for a proof). With such an instantiation, our entire construction of lossless expanders becomes strongly explicit. However, for simplicity in this paper we primarily discuss ordinary (weak) explicitness.
\end{remark}

\begin{remark}
  \label{remark:robustexp}
  The \ref{it:bipramproof}nd proof of Proposition~\ref{prop:nonlazyexp} is robust in the sense that it will still work for onesided near-Ramanujan bipartite graphs, that is, for $(k,D)$-biregular bipartite graphs for which the nontrivial eigenvalues of the adjacency matrix have absolute value $\leq\sqrt{D-1}+\sqrt{k-1}+\alpha$, as long as $\alpha<o(\sqrt{D})$. It for instance follows that this argument works for random bipartite graphs, which achieve $\alpha=o(1)$ \cite{brito_spectral_2022}. In constrast, the analysis of unique-neighbor expansion in \cite{asherov_bipartite_2023} requires exactly Ramanujan bipartite graphs with twosided expansion, meaning they require all nontrivial eigenvalues to have absolute value lying in the interval $[\sqrt{D-1}-\sqrt{k-1},\sqrt{D-1}+\sqrt{k-1}]$.
\end{remark}



We will make use of the following bound showing lossless expansion of random bipartite graphs.

\begin{proposition}[Well known; see for instance Theorem 11.2.8 of \cite{guruswami_essential_2022}]
  \label{prop:randloss}
  For all constants $\beta,\epsilon>0$, there exists an integer $d=d(\beta,\epsilon)=\Theta(\log(1/\epsilon\beta)/\epsilon)$, a sufficiently large integer $n_0=n_0(\beta,\epsilon)$, and a sufficiently small real number $\mu=\mu(\beta,\epsilon)=\Theta(\epsilon\beta/d)$ such that for all $n\geq n_0$, there exists a bipartite graph $G$ with left-degree $d$ and with $|L(G)|=n$, $|R(G)|=\lfloor\beta n\rfloor$ such that $G$ is a $(\mu,\epsilon)$-lossless expander.
\end{proposition}





\section{Construction}
\label{sec:construct}
In this section, we present our construction that we use to prove Theorem~\ref{thm:expanders}. We will subsequently prove that this construction has lossless expansion in Section~\ref{sec:losslessproof}. Our construction uses essentially the same framework as the unique neighbor expanders of \cite{asherov_bipartite_2023}, though we analyze it differently to obtain the stronger object of lossless expanders.

\begin{figure}
  \centering
  \begin{tikzpicture}[thick,
    every node/.style={draw,circle},
    fsnode/.style={fill=black},
    ssnode/.style={fill=black},
    every fit/.style={ellipse,draw},
    ]

    \begin{scope}[start chain=going below,node distance=.3cm]
      \foreach \i in {1,2,...,12}
      \node[fsnode,on chain] (f\i) {};
    \end{scope}

    \begin{scope}[xshift=4cm,yshift=-.5cm,start chain=going below,node distance=.3cm]
      \foreach \i in {1,...,3}
      \node[ssnode,on chain] (s\i) {};
    \end{scope}
    \begin{scope}[xshift=4cm,yshift=-3cm,start chain=going below,node distance=.3cm]
      \foreach \i in {4,...,6}
      \node[ssnode,on chain] (s\i) {};
    \end{scope}
    \begin{scope}[xshift=4cm,yshift=-5.5cm,start chain=going below,node distance=.3cm]
      \foreach \i in {7,...,9}
      \node[ssnode,on chain] (s\i) {};
    \end{scope}

    \node [black,fit=(f3) (f8),inner sep=-14pt,text width=2cm,label=left:$N_X(v)$] {};
    \node [black,fit=(s4) (s6),inner sep=-5pt,text width=1cm,label=right:$\{v\}\times R(G_0)$] {};

    \draw (f3) -- (s5);
    \draw (f3) -- (s6);
    \draw (f4) -- (s4);
    \draw (f4) -- (s5);
    \draw (f5) -- (s4);
    \draw (f5) -- (s5);

    \draw (f6) -- (s4);
    \draw (f6) -- (s6);
    \draw (f7) -- (s5);
    \draw (f7) -- (s6);
    \draw (f8) -- (s4);
    \draw (f8) -- (s6);
    
    \node[fill=white,draw=none,rectangle,minimum width=1.3cm,minimum height=4cm] at (2, -3.5) {$G_0^v$};
  \end{tikzpicture}
  \caption{\label{fig:construction}An illustration of how we construct our lossless expanders $G=G(X,G_0)$ as the union over $v\in R(X)$ of the gadgets $G_0^v\cong G_0$. A single such gadget is shown above, connecting the left vertices $L(G_0^v)=N_X(v)\cong L(G_0)$ to the right vertices $R(G_0^v)=\{v\}\times R(G_0)\cong R(G_0)$.}
\end{figure}
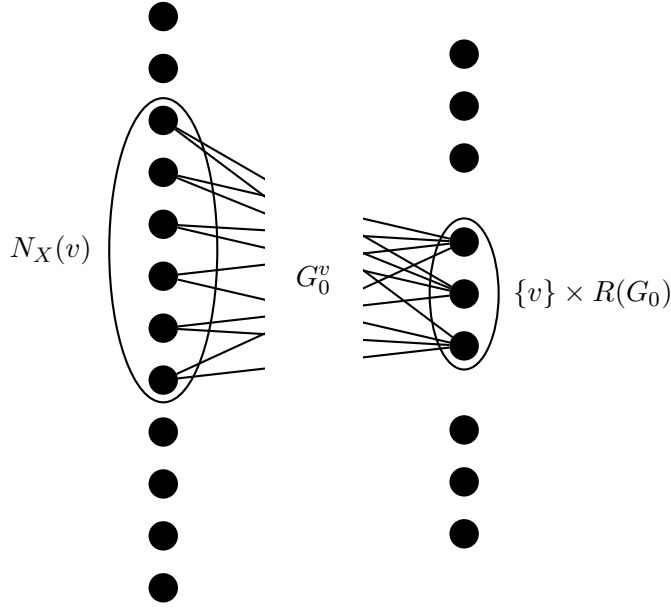

\subsection{General framework}
We first describe the general framework for constructing our lossless expanders $G$, and then present precise parameters. Throughout this section we fix a constant interval $\beta=(\beta^{(1)},\beta^{(2)})\subseteq\bR_+$ inside which we want the $|R(G)|/|L(G)|$ to lie, and we fix $\epsilon>0$ denoting the desired expansion.

To construct $G$, we begin by taking a $(k,D_0)$-biregular graph $X$ for which the nonlazy square $X'$ has good spectral expansion $\lambda_2(X')$. We choose $X$ using Proposition~\ref{prop:nonlazyexp}, so we think of $|V(X)|$ growing arbitrarily large for fixed $k\ll D_0$ and fixed $\lambda_2(X')=\poly(1/k,\beta^{(2)})$.

We also choose a constant-sized bipartite ``gadget'' graph $G_0$ with $|L(G_0)|=D_0$ and $|R(G_0)|=\lfloor D_0\beta^{(2)}/k\rfloor$ that is a lossless expander, as guaranteed to exist by Proposition~\ref{prop:randloss}.

We now define our desired lossless expander $G=G(X,G_0)$ as follows. An illustration is provided in Figure~\ref{fig:construction}.
\begin{itemize}
\item The left vertex set $L(G)=L(X)$ of $G$ equals the left vertex set of $X$.
\item The right vertex set $R(G)=R(X)\times R(G_0)$ of $G$ is obtained by replacing each right vertex of $X$ with a cluster of $|R(G_0)|$ vertices.
\item The edge set $E(G)$ is defined as follows. For each $v\in R(X)$, as $|N_X(v)|=D_0=|L(G_0)|$, we may associate the neighborhood $N_X(v)$ with $L(G_0)$. We may similarly associate $\{v\}\times R(G_0)\subseteq R(G)$ with $R(G_0)$. Therefore we may introduce a copy $G_0^v$ of $G_0$ with left vertex set $L(G_0^v):=N_X(v)\subseteq L(G)$ and right vertex set $R(G_0^v):=\{v\}\times R(G_0)\subseteq R(G)$. We then let $G$ be the union of the graphs $G_0^v$ over all $v\in R(X)$. That is, $(w,(v,v_0))\in E(G)$ if and only if $(w,(v,v_0))\in E(G_0^v)$.
\end{itemize}

The resulting graph $G=G(X,G_0):=\sum_{v\in L(X)}G_0^v$ satisfies the following basic properties.

\begin{claim}
  \label{claim:leftdeg}
  If $G_0$ has left-degree $d_0$, then $G$ has left-degree $D=d_0k$.
\end{claim}
\begin{proof}
  Each vertex $w\in L(G)=L(X)$ has $k$ $X$-neighbors $v\in N_X(w)\subseteq R(X)$, for each of which the graph $G_0^v$ contributes $d_0$ edges to $w$ in $G$.
\end{proof}

\begin{claim}
  \label{claim:balance}
  If $D_0\geq k/(\beta^{(2)}-\beta^{(1)})$, then $|R(G)|/|L(G)|\in(\beta^{(1)},\beta^{(2)})$.
\end{claim}
\begin{proof}
  By construction
  \begin{equation*}
    \frac{|R(G)|}{|L(G)|} = \frac{|R(X)|\cdot|R(G_0)|}{|L(X)|} = \frac{k}{D_0}\cdot\left\lfloor\frac{D_0}{k}\cdot\beta^{(2)}\right\rfloor,
  \end{equation*}
  which is at most $\beta^{(2)}$ and at least $k/D_0\cdot(D_0/k\cdot\beta^{(2)}-1)=\beta^{(2)}-k/D_0$. Thus the claim follows.
\end{proof}

\subsection{Choosing the parameters}
\label{sec:param}
Formally, our construction uses the following parameters and components, for fixed $\beta=(\beta^{(1)},\beta^{(2)})\subseteq\bR_+$ and $\epsilon>0$:
\begin{itemize}
\item Let $k=\lceil 10/\epsilon\rceil$.
\item For balance constant $\beta_0=\beta^{(2)}/k$ and loss constant $\epsilon_0=\epsilon/10$, let $d_0=d_0(\beta_0,\epsilon_0)=\Theta(\log(k/\beta^{(2)}\epsilon)/\epsilon)$, $n_0=n_0(\beta_0,\epsilon_0)$, and $\mu_0=\mu_0(\beta_0,\epsilon_0)=\Theta(\epsilon\beta^{(2)}/d_0k)$ be the degree, size bound, and relative set size bound respectively given by Proposition~\ref{prop:randloss}.
\item Let $\lambda_2=\mu_0/10k^3$.
\item Let $D_0\geq \max\{n_0,k/(\beta^{(2)}-\beta^{(1)})\}$ be an integer such that there exists an infinite explicit family $\cX$ of $(k,D_0)$-biregular bipartite graphs $X\in\cX$ for which the nonlazy square $X'$ has $\lambda_2(X')\leq\lambda_2$. Such a $D_0$ exists by Proposition~\ref{prop:nonlazyexp}.
\item Let $G_0$ be a bipartite graph with left-degree $d_0$ and with $|L(G_0)|=D_0$, $|R(G_0)|=\lfloor\beta_0D_0\rfloor$ that is a $(\mu_0,\epsilon_0)$-lossless expander, as given by Proposition~\ref{prop:randloss}.
\end{itemize}

Our desired family $\cG$ of graphs is then defined as $\cG=\{G(X,G_0):X\in\cX\}$, where $G=G(X,G_0)$ is constructed from $X$ and $G_0$ as described above.

The following explicitness claim is immediate from our construction, as $G=G(X,G_0)$ is by definition obtained from $X$ by inserting vertices and edges locally within the (constant-sized) neighborhoods of vertices in $X$.

\begin{claim}
  \label{claim:explicit}
  If the family $\cX$ is (strongly) explicit, then the family $\cG$ is (strongly) explicit.
\end{claim}

Note that $\cX$ can be made strongly explicit by using a strongly explicit family of high-dimensional expanders (e.g.~\cite{kaufman_construction_2018}) as in the \ref{it:hdxproof}st proof of Proposition~\ref{prop:nonlazyexp}.


\section{Proof of lossless expansion}
\label{sec:losslessproof}
We now show that the graphs $G\in\cG$ defined in Section~\ref{sec:construct} have lossless expansion. Specifically, we show the following result, which when combined with Claim~\ref{claim:leftdeg}, Claim~\ref{claim:balance}, and Claim~\ref{claim:explicit} directly implies Theorem~\ref{thm:expanders}.

\begin{proposition}
  \label{prop:expformal}
  Defining all variables as in Section~\ref{sec:construct}, then for every $X\in\cX$, the bipartite graph $G=G(X,G_0)\in\cG$ is a $(\mu,\epsilon)$-lossless expander for $\mu=k^2\lambda_2^2$.
\end{proposition}

\begin{proof}
  Fix any set $S\subseteq L(G)=L(X)$ of size $|S|\leq\mu|L(G)|$. Define the ``heavy vertices'' $H=\{v\in R(X):|N_X(v)\cap S|\geq\mu_0D_0\}\subseteq R(X)$ to be those vertices in $R(X)$ incident to $\geq\mu_0D_0$ vertices in $S$. Below we present the key claim for our proof, which states that most vertices in $S$ are incident to $\leq 1$ heavy vertices, and therefore to $\geq k-1$ non-heavy vertices. We prove this claim with an application of the expander mixing lemma.

  We first need the following notation. For $0\leq i\leq k$, let $S_{=i}=\{w\in L(X):|N_X(w)\cap H|=i\}\subseteq L(X)$ be the set of vertices in $L(X)$ incident to exactly $i$ heavy vertices. Similary define $S_{\geq i}=\bigcup_{j\geq i}S_{=j}$ and $S_{\leq i}=\bigcup_{j\leq i}S_{=j}$.

  \begin{claim}
    \label{claim:Sleq1}
    It holds that
    \begin{align*}
      \frac{|S_{\leq 1}|}{|S|}
      &\geq 1-\frac{1}{5k}.
    \end{align*}
  \end{claim}
  \begin{proof}[Proof of Claim~\ref{claim:Sleq1}]
    By definition
    \begin{align*}
      |H|
      &\leq \frac{k|S|}{\mu_0D_0} = \frac{|S|}{10k^2\lambda_2D_0} \leq \frac{\mu|L(G)|}{10k^2\lambda_2D_0} = \frac{\lambda_2|R(X)|}{10},
    \end{align*}
    where the equalities above apply the definitions of $\lambda_2$ and $\mu$ respectively. Thus letting $X'$ be the nonlazy square of $X$, Lemma~\ref{lem:expmix} implies that
    \begin{equation*}
      w_{X'}(H,H) \leq \frac{11}{10}\lambda_2D_0(k-1)|H|.
    \end{equation*}

    For each vertex $w\in S_{\geq 2}$, we may choose two distinct heavy vertices $v,v'\in N_X(w)\cap H$, and let $e(w)=\{v,v'\}\in E(X')$ be the edge in $X'$ induced by the path $v\rightarrow w\rightarrow v'$ in $X$. By definition all edges $e(w)$ for $w\in S_{\geq 2}$ are distinct,\footnote{Here we view $E(X')$ as a multiset where the edge $(v,v')$ has a distinct copy for every path $v\rightarrow w\rightarrow v'$ in $X$.} and both endpoints of each $e(w)$ lie in $H$, so
    \begin{equation*}
      |S_{\geq 2}| = |\{e(w):w\in S_{\geq 2}\}| \leq w_{X'}(H,H) \leq \frac{11}{10}\lambda_2D_0(k-1)|H|.
    \end{equation*}
    Meanwhile, as each vertex $w\in S_{\geq 1}$ is incident to at least one but at most $k$ heavy vertices, we have that
    \begin{equation*}
      |S_{\geq 1}| \geq \frac{\mu_0D_0|H|}{k}.
    \end{equation*}
    Thus
    \begin{align*}
      |S_{\leq 1}|
      &\geq |S_{=1}| \\
      &= |S_{\geq 1}|-|S_{\geq 2}| \\
      &\geq \left(\frac{\mu_0}{k}-\frac{11}{10}\lambda_2(k-1)\right)D_0|H| \\
      &\geq 8k^2\lambda_2D_0|H| \\
      &\geq 5k|S_{\geq 2}|.
    \end{align*}
    where the third inequality above applies the definition of $\lambda_2$. Thus
    \begin{align*}
      \frac{|S_{\leq 1}|}{|S|}
      &= \frac{|S_{\leq 1}|}{|S_{\leq 1}|+|S_{\geq 2}|} \geq 1-\frac{|S_{\geq 2}|}{|S_{\leq 1}|} \geq 1-\frac{1}{5k}.
    \end{align*}
  \end{proof}

  We now prove the proposition using this fact that most vertices $w\in S$ are incident to at most one heavy vertex. By definition $G_0$ is a $(\mu_0,\epsilon_0)$-lossless expander, so the intersection of $S$ with the $X$-neighborhood $N_X(v)\subseteq L(X)=L(G)$ of every non-heavy vertex $v\in R(X)\setminus H$ exhibits expansion $(1-\epsilon_0)d_0$ in the subgraph $G_0^v\cong G_0$ of $G$. Then it follows that most vertices in $L(G)$ contribute $(k-1)(1-\epsilon_0)d_0$ to the expansion of $S$ in $G$, which implies the proposition.

  Formally, for $v\in R(X)\setminus H$ then by definition $|N_X(v)\cap S|\leq\mu_0D_0$, so
  \begin{equation*}
    |N_{G_0^v}(N_X(v)\cap S)| \geq (1-\epsilon_0)d_0|N_X(v)\cap S|.
  \end{equation*}
  Thus
  \begin{align*}
    |N_G(S)|
    &= \sum_{v\in R(X)}|N_{G_0^v}(N_X(v)\cap S)| \\
    &\geq \sum_{v\in R(X)\setminus H}(1-\epsilon_0)d_0|N_X(v)\cap S| \\
    &= \sum_{w\in S}\sum_{v\in N_X(w)\setminus H}(1-\epsilon_0)d_0 \\
    &\geq \sum_{w\in S_{\leq 1}}(k-1)(1-\epsilon_0)d_0 \\
    &\geq |S|\left(1-\frac{1}{5k}\right)(k-1)(1-\epsilon_0)d_0 \\
    &\geq |S|\left(1-\frac{\epsilon}{50}\right)\left(1-\frac{\epsilon}{10}\right)k\left(1-\frac{\epsilon}{10}\right)d_0 \\
    &\geq |S|(1-\epsilon)D,
  \end{align*}
  where the third inequality holds by Claim~\ref{claim:Sleq1}, and the final inequality holds because $D=kd_0$ by Claim~\ref{claim:leftdeg}. Thus we have shown that $G$ is a $(\mu,\epsilon)$-lossless expander, as desired.
\end{proof}

\begin{remark}
  One could hope that our construction $G=G(X,G_0)$ happens to expand losslessly from right to left as well. However, if we fix any $v\in R(X)$, then the set $T=\{v\}\times R(G_0)\subseteq R(G)$ consisting of a single cluster of right vertices in $G$ has neighborhood of size $|N_G(T)|=|N_X(v)|=D_0$, whereas $(\Omega(1),\epsilon)$-lossless right-to-left expansion would require the much larger neighborhood size $|N_G(T)|\geq(1-\epsilon)D_0d_0$. Thus a new approach is needed for twosided expansion.
\end{remark}

\begin{remark}
  \label{remark:compareparam}
  Given $\epsilon>0$ and $\beta^{(2)}<1/2$, by tracing through the parameters in Section~\ref{sec:param}, we find that our $(\mu,\epsilon)$-lossless expanders in Proposition~\ref{prop:expformal} have degree
  \begin{align*}
    D &= O\left(\frac{\log\frac{1}{\epsilon}+\log\frac{1}{\beta^{(2)}}}{\epsilon^2}\right)
  \end{align*}
  and exhibit expansion up to the cutoff
  \begin{align*}
    \mu &= \Omega\left(\frac{\epsilon^{10}{\beta^{(2)}}^2}{\left(\log\frac{1}{\epsilon}+\log\frac{1}{\beta^{(2)}}\right)^2}\right).
  \end{align*}
  In comparison, the construction of \cite{capalbo_randomness_2002} has degree
  \begin{align*}
    D' &= O\left(\frac{\left(\log\frac{1}{\epsilon}+\log\frac{1}{\beta^{(2)}}\right)^2}{\epsilon^3}\right)
  \end{align*}
  and exhibits expansion up to the cutoff
  \begin{align*}
    \mu' &= \Theta\left(\frac{\epsilon\beta^{(2)}}{D'}\right) = \Omega\left(\frac{\epsilon^4\beta^{(2)}}{\left(\log\frac{1}{\epsilon}+\log\frac{1}{\beta^{(2)}}\right)^2}\right).
  \end{align*}
  Thus the two constructions achieve incomparable parameters; we achieve a smaller degree $D<D'$, whereas \cite{capalbo_randomness_2002} achieves a larger expansion cutoff $\mu'>\mu$.
\end{remark}

\section{Acknowledgments}
The author thanks Omar Alrabiah, Venkatesan Guruswami, Sidhanth Mohanty, Christopher A.~Pattison, and Salil Vadhan for numerous helpful discussions and suggestions, and for helping improve the exposition. S.~Mohanty suggested the \ref{it:bipramproof}nd proof of Proposition~\ref{prop:nonlazyexp} using bipartite Ramanujan graphs. The author thanks Peter Manohar and Justin Oh for pointing out applications of lossless expanders with improved parameters.

\bibliography{library}
\bibliographystyle{alpha}

\end{document}